\documentclass[12pt,english,reqno]{amsart}
\usepackage[T1]{fontenc}
\usepackage[latin9]{inputenc}
\usepackage{verbatim}
\usepackage{amstext}
\usepackage{amsthm}

\makeatletter
\numberwithin{equation}{section}
\numberwithin{figure}{section}
\theoremstyle{plain}
\newtheorem{thm}{\protect\theoremname}[section]
  \theoremstyle{definition}
  \newtheorem{defn}[thm]{\protect\definitionname}
  \theoremstyle{plain}
  \newtheorem{lem}[thm]{\protect\lemmaname}
  \theoremstyle{remark}
  
  \theoremstyle{plain}
  \newtheorem{prop}[thm]{\protect\propositionname}
  \theoremstyle{plain}
  \newtheorem{cor}[thm]{\protect\corollaryname}

\usepackage{amscd}

\newtheorem{theorem}{Theorem}[section]
\theoremstyle{plain}

\newtheorem{remark}[theorem]{Remark}

\newcommand{\eqdef}{\stackrel{\mathrm{def}}{=}}

\newcommand{\N}{\mathbb{N}}
\newcommand{\R}{\mathbb{R}}

\newcommand{\Z}{\mathbb{Z}}
\newcommand{\wlim}{\operatorname{w-lim}}

\makeatother

\usepackage{babel}
  \providecommand{\corollaryname}{Corollary}
  \providecommand{\definitionname}{Definition}
  \providecommand{\lemmaname}{Lemma}
  \providecommand{\propositionname}{Proposition}
  \providecommand{\remarkname}{Remark}
\providecommand{\theoremname}{Theorem}

\begin{document}

\title[Profile decomposition]{Profile decomposition of Struwe-Solimini for manifolds with bounded geometry}

\author{Kunnath Sandeep}

\address{TIFR Centre for Applicable Mathematics, Post Bag no. 6503,Sharada nagar,Yelahanka New Town, Bangalore-560065,India.}

\email{sandeep@math.tifrbng.res.in}

\author{Cyril Tintarev}

\thanks{One of the authors (C.T.) expresses his gratitude to Professor Sandeep
and the faculty of TIFR-CAM for their warm hospitality.}

\address{Sankt Olofsgatan 66B, 75 330 Uppsala, Sweden}

\email{tammouz@gmail.com}

\begin{abstract}
	For many known non-compact embeddings of two Banach
	spaces $E\hookrightarrow F$, every bounded sequence  in $E$  has a subsequence that takes form of  a \emph{profile decomposition} - a sum of clearly structured terms with asymptotically disjoint supports plus a remainder that vanishes in the norm of $F$. 
	In this paper we construct a profile decomposition for arbitrary sequences in the Sobolev space $H^{1,2}(M)$ of a Riemannian manifold with bounded geometry, relative to the embedding of $H^{1,2}(M)$ into $L^{2^*}(M)$, generalizing the well-known profile decomposition of Struwe \cite[Proposition 2.1]{Struwe} to the case of any bounded sequence and a non-compact manifold.
\end{abstract}
\maketitle

\section{Introduction}

\emph{Defect of compactness}, relative to an embedding of two Banach
spaces $E\hookrightarrow F$, is a difference $u_{k}-u$ between a
weakly convergent sequence $u_{k}\rightharpoonup u$ in $E$ and its
limit, taken up to a remainder that vanishes in the norm of $F$.
In particular, if the embedding is compact and $E$ is reflexive,
the defect of compactness is null. For many embeddings there exist
well-structured descriptions of the defect of compactness known
as \emph{profile decompositions}. In particular, profile decompositions
relative to Sobolev embeddings are sums of terms with asymptotically
disjoint supports, which are called \emph{elementary concentrations}
or \emph{bubbles. }Profile decompositions were motivated by studies
of concentration phenomena in the early 1980's by Uhlenbeck, Brezis,
Coron, Nirenberg, Aubin and Lions, and they play significant role
in verification of convergence of functional sequences in applied
analysis, particularly when the information available via the classical
concentration compactness method is not enough detailed. 

Profile decompositions are known when the embedding $E\hookrightarrow F$
is \emph{cocompact} relative to some group $\mathcal{G}$ of bijective
isometries on $E$ (embedding $E\hookrightarrow F$ is called $\mathcal{G}$
-cocompact if any sequence $(u_{k})$ in $E$ satisfying $g_{k}u_{k}\rightharpoonup0$
for any sequence of operators $(g_{k})$ in $\mathcal{G}$ vanishes
in the norm of $F$). It is easy to show (the example is due to Jaffard, \cite{Jaffard}) that $\ell^{\infty}(\Z)$
is cocompactly embedded into itself relative to the group of shifts
$\mathcal{G}=\lbrace(a_{n})\mapsto(a_{n+m})\rbrace_{m\in\Z}$: $g_{k}u_{k}\rightharpoonup0$
means in this case uniform convergence. The earliest cocompactness
result for functional spaces known to the authors is the proof of cocompactness
of embedding of the inhomogeneous Sobolev space $H^{1,p}(\R^{N})$,
$N>p$, into $L^{q}$, $q\in(p,p^{*})$, where $p^{*}=\frac{pN}{N-p}$,
relative to the group of shifts $u\mapsto u(\cdot-y),\;y\in\R^{N}$,
by E. Lieb \cite{LIeb}, although expressed in very different terms
(the term \emph{cocompactness} appeared in literature only a decade
ago). 

A profile decomposition relative to a group $\mathcal{G}$ of bijective
isometries expresses the defect of compactness as a sum of \emph{elementary
concentrations}, or \emph{bubbles}, $\sum_{n\in\N}g_{k}^{(n)}w^{(n)}$
, with some $g_{k}^{(n)}\in\mathcal{G}$ and $w^{(n)}\in E$, $k\in\N$,
$n\in\N$. The elements $w^{(n)}$, called \emph{concentration profiles},
are then obtained as weak limits of $(g_{k}^{(n)})^{-1}u_{k}$ as
$k\to\infty$. Typical examples of groups $\mathcal{G}$, involved
in profile decompositions, are the above mentioned group of shifts
and the rescaling group, which is a product group of shifts and dilations
$u\mapsto t^{r}u(t\cdot)$, $t>0$, where $r=\frac{N-p}{p}$ for $\dot{H}^{1,p}(\R^{N})$,
$N>p$. 
\vskip3mm
For a smooth Riemannian manifold $M$, the Sobolev space $H^{1,2}(M)$
is defined as a completion of $C_{0}^{\infty}(M)$ with respect to
the norm 
\[
\|u\|^{2}=\int_{M}(|du|^{2}+|u|^{2})\mathrm{d}\mu,
\]
 where $d$ stays for the covariant derivative and $\mu$ for the
Riemannian measure on $M$. We will discuss later also the matters
concerning definition of a counterpart of the space $\dot{H}^{1,2}(\R^{N})$
for manifolds. In what follows the unqualified notation of the norm
will always refer to the $H^{1,2}(M)$-norm, and this extends to the notation of the corresponding scalar product. The $L^{p}$ norm for
the manifold will be denoted as $\|u\|_{p}$.

This paper is motivated by a profile decomposition for sequences in Sobolev spaces of compact manifolds by Struwe \cite[Proposition 2.1]{Struwe},
where the defect of compactness for the embedding $H^{1,2}(M)\hookrightarrow L^{2^{*}}(M)$
is a finite sum of bubbles of the form $t_{k}^{\frac{N-2}{2}}w\circ\exp_{y_k}^{-1}(t_{k}\exp_{y_k}(x))$,
$t_{k}\to\infty$, $y_k\in M$, multiplied by a suitable cut-off function. 
Struwe's result is restricted to  Palais-Smale sequences of
semilinear elliptic functionals, which assures that the number of bubbles is finite. Struwe's argument for vanishing of the remainder in $L^{2^{*}}$ is also based on properties of Palais-Smale sequences. 

The earliest known result on profile decompositions for general bounded
sequences in $\dot{H}^{1,p}(\R^{N})$ equipped with the rescaling
group was proved by Solimini \cite{Solimini}. It was repeated later,
and with a weaker form of asymptotics, in \cite{Gerard} and \cite{Jaffard}
(\cite{Jaffard} also extended the result to fractional Sobolev spaces).
Subsequently, profile decompositions were found in the general functional-analytic
setting relative to a general group of isometries, for Hilbert spaces
in \cite{SchinTin}, and for uniformly convex Banach spaces with the
Opial condition in \cite{SoliTi} (without the Opial condition weak
convergence has to be replaced in the statement by the less-known
Delta convergence). However, despite the general character of this
result, some known profile decompositions do not follow directly from
\cite{SoliTi}, in particular, when the space $E$ is not reflexive
(\cite{AT_BV}), when one has only a semigroup of isometries (e.g.
\cite{AdiTi-Pisa}), or when the profile decomposition can be expressed
without a group (\cite{DevT,SkrTi4}). The latter papers are dealing
with generalizations of the profile decomposition of Solimini to the
case of Riemannian manifolds. Paper \cite{SkrTi4} proves profile
decomposition for the embedding $H^{1,2}(M)\hookrightarrow L^{p}(M)$,
$2<p<2^{*}$, when $M$ is a non-compact manifold of bounded geometry,
where the loss of compactness is due to bubbles shifting to infinity
(concentration compactness of type I in the framework of Lions), and
paper \cite{DevT} deals with the case of compact manifold and $p=2^{*}$,
where the loss of compactness is due to Struwe's bubbles, thus generalizing
Struwe's profile decomposition to the case of general sequences. (This
corresponds to concentration compactness of type II in the sense of
Lions). However, the results of \cite{DevT} and \cite{SkrTi4} do
not generalize the original profile decomposition for $\R^{N}$
to the case of manifolds in full: the former is restricted to compact manifolds
and the latter is restricted to the case of subcritical exponent.
The present paper gives an answer to the question if there is a profile
decomposition for the embedding $H^{1,2}(M)\hookrightarrow L^{2^{*}}(M)$
with both shifts and Struwe's bubbles. Our method of proof is different from  \cite{DevT}, whose argument does not yield an obvious generalisation to non-compact manifolds, while the result of  \cite{SkrTi4} is employed here as an intermediate step.

Certain limitations have to be imposed on $M$, since the embedding
$H^{1,2}(M)\hookrightarrow L^{2^{*}}(M)$ itself does not exist for
every manifold. Following the reasoning of \cite{SkrTi4} we deal
here only with manifolds of bounded geometry, which generate
concentration profiles at infinty which are still functions in a Sobolev space of a smooth Riemannian manifold, although this manifold is generally not the original $M$.

The other limitation concerns the use of the full Sobolev norm and
not the gradient norm of $\dot{H}^{1,2}$ as in \cite{Solimini}.
For sequences bounded in $H^{1,2}(\R^{N})$, the $L^{2}$-bound eliminates
from the profile decomposition of Solimini any deflations (``reverse
bubbles'') $t_{k}^{\frac{N-2}{2}}w(t_{k}(x-y_{k}))$ with $t_{k}\to0$, because their $L^{2}$-norms are unbounded. In comparison,
on the hyperbolic space (or any space of negative curvature bound
away from zero, see \cite{chavel-eigen}) the gradient norm dominates the $L^{2}$-norm,
which suppresses appearance of deflations in the profile decomposition. On the other hand,
on the sphere (as on any other compact manifold) quadratic form $\int_{M}|du|^{2}d\mu$
defines only a seminorm whose null space consists of constants. Appending it, for example, by $\left(\int_{M}u\,d\mu\right)^{2}$ yields a norm that also dominates the $L^{2}$-norm which eliminates deflations. 

The example of $M=\R^N$ shows that when $\dot{H}^{1,2}(M)\neq H^{1,2}(M)$, the defect of compactness may have to account also for some analog of "deflation" terms $t_k^\frac{N-2}{2}w(t_k\cdot)$, $t_k\to 0$, but while every manifold admits local "zoom-in" maps, there is no natural global "zoom out" map that could describe concentration of deflative character on every (approximately flat) manifold.  
\vskip3mm
The paper is organized as follows. In Section 2 we show that sequences
bounded in $H^{1,2}(M)$, that already vanish in $L^{p}(M)$ for some
$p\in(2,2^{*})$, vanish also in $L^{2^{*}}(M)$ if they satisfy a
``no bubbles'' condition. In Section 3 we show that sequences bounded
in $H^{1,2}(M)$ that vanish in $L^{p}(M)$ for some $p\in(2,2^{*})$,
have a subsequence that satisfies the ``no bubbles'' condition (and
thus vanishes in $L^{2^{*}}(M)$ ) after one subtracts from it all
its bubbles, and that the bubbles are asymptotically decoupled. In Section 4 we recall the result of \cite{SkrTi4} on the defect of compactness
for the embedding $H^{1,2}(M)\hookrightarrow L^{p}(M)$, $p\in(2,2^{*})$,
which says that a sequence bounded in $H^{1,2}(M)$ has a subsequence
that, after subtraction of its weak limit and all suitably defined
shift-concentration terms, vanishes in $L^{p}(M)$,$p\in(2,2^{*})$.
Combining this result with the result of Section 3 we obtain the main
result of this paper, Theorem \ref{thm:finalPD}, which says that the defect of compactness
for the embedding $H^{1,2}(M)\hookrightarrow L^{2^*}(M)$
is a sum of a series of bubbles and a series of shift-concentrations,
all of them mutually decoupled. 
Appendix summarizes some statements concerning manifolds of bounded geometry and manifolds at infinity used in the main body of the paper.

\section{Vanishing lemma}

From now on we assume that $M$ is a smooth, complete, connected $N$-dimensional
Riemannian manifold of bounded geometry. The latter property is defined
as follows. 
\begin{defn}
\label{def:bg}(Definition A.1.1 from \cite{Shubin}) A smooth Riemannian
manifold $M$ is of bounded geometry if the following two conditions
are satisfied: 

(i) The injectivity radius $r(M)$ of $M$ is positive. 

(ii) Every covariant derivative of the Riemann curvature tensor $R^{M}$of
M is bounded, i.e., $\nabla^{k}R^{M}\in L^{\infty}(M)$ for every
$k=0,1,\dots$ 
\end{defn}

We refer the reader to the appendix for elementary properties of manifolds
of bounded geometry used in this paper, and existence of an appropriate
covering of such manifolds. 

In what follows $B(x,r)$ will denote a geodesic ball in $M$ and
$\Omega_{r}$ will denote the ball in $\R^{N}$ of radius $r$ centered
at the origin. Let $e_{x}:\Omega_{r}\to B(x,r)$ , $r<r(M)$. be an
exponential map of $M$ under identification of the ball of radius
$r$ centered at the origin in $T_{x}M$ as the ball $\Omega_{r}\subset\R^{N}$
(we reserve the standard notation $\mathrm{exp}_{x}$ for a standard
exponential map $T_{x}M\to M$, so that $e_{x}=\exp_{x}\circ i_{x}$
where $i_{x}$ is an arbitrarily fixed linear isometry between $\R^{N}$
and $T_{x}M$). 

We will deal first with sequences that are bounded in $H^{1,2}(M)$
and vanish in $L^{p}(M)$ for some $p\in(2,2^{*})$ (note that if
this is the case for some value of $p$ in the interval than it is
true for any other value in the interval). Such sequences may still
have $L^{2^{*}}$-norm bounded away from zero. For example, let $v\in C_{0}^{\infty}(\Omega_{r})$
and let $u_{t}(x)=\sum_{n=1}^{m}t^{\frac{N-2}{2}}v(te_{y_{n}}^{-1}(x))$,
where $y_{1},\dots,y_{m}$ are distinct points on $M$, $m\in\N$,
$t>1$, and each term in the sum is understood as extended by zero
outside of $B(y_{n},\Omega_{r/t}$). Then $u_{t}\to0$ in $L^{p}(M)$
for any $p\in(2,2^{*}),$ while the sequence is bounded in the $H^{1,2}$-norm
and $\|u_{k}\|_{2^{*}}^{2^{*}}\to m\|v\|_{2^{*}}^{2^{*}}$. If we
provide a profile decomposition for such sequence with the remainder
vanishing in $L^{p}(M)$ for all $p\in(2,2^{*}],$ then we will have
a profile decomposition for all sequences relative to the embedding
$H^{1,2}(M)\hookrightarrow L^{2^{*}}(M)$, consisting of the sum of
the profile decomposition relative to the embedding $H^{1,2}(M)\hookrightarrow L^{p}(M)$,
$p\in(2,2^{*})$ (given in \cite{SkrTi4}) and our new profile decomposition
for this remainder (which vanishes in $L^{p}(M)).$ For these ends
we identify a sequence that vanishes in $L^{p}(M)$ and passes a ``nonconcentration''
test (relation (\ref{eq:noconc}) below) as vanishing in $L^{2^{*}}(M).$
\begin{remark}
 In what follows we will often consider sequences of the form  $t_{k}^{\frac{N-2}{2}} u_{k}\circ e_{y_{k}}(t_{k}\cdot)$, $t_k\to 0$, which are defined on $\Omega_R\subset\R^N$ for every $R$, provided that $k$, dependently on $R$, is sufficiently large, which allows a natural definition of their pointwise limit on $\R^N$. It is easy to see that if $(u_k)$ is a bounded sequence in $H^{1,2}(M)$, then $(t_{k}^{\frac{N-2}{2}} u_{k}\circ e_{y_{k}}(t_{k}\cdot))$ converges pointwise on $\R^N$ to some $w$ if and only if for any $\varphi\in C_c(\R^N)$, $\int_Md(t_{k}^{\frac{N-2}{2}} u_{k}\circ e_{y_{k}}(t_{k}\cdot))\cdot d\varphi d\mu\to \int_Mdw\cdot d\varphi d\mu$ as $k\to\infty$, which we for short will denote, slightly abusing the notation of weak convergence, as $u_k\rightharpoonup w$ or $w=\wlim u_k$. Moreover, for every sequence $(u_k)$ bounded in $H^{1,2}(M)$ and every sequence $(t_k)$ with $t_k\to 0$, the sequence   $(t_{k}^{\frac{N-2}{2}} u_{k}\circ e_{y_{k}}(t_{k}\cdot))$  has a subsequence that converges pointwise and weakly in the above sense.  
\end{remark}
From now on we fix a positive number $r<r(M)$. 
\begin{thm}["Vanishing lemma"]
\label{thm:vanishing} Let $M$ be a complete smooth $N$-dimensional
Riemannian manifold of bounded geometry, and let $(u_{k})$ be a bounded
sequence in $H^{1,2}(M)$ such that for any sequence of points $y_{k}\in M$,
and any sequence $t_{k}\to0$,

\begin{equation}
t_{k}^{\frac{N-2}{2}}u_{k}\circ e_{y_{k}}(t_{k}\xi)\longrightarrow0\mbox{ a.e. in }\R^{N}.\label{eq:noconc}
\end{equation}
Furthermore, assume that $u_{k}\to0$ in $L^{p}(M)$ for some $p\in(2,2^{*})$.
Then $u_{k}\to0$ in $L^{2^{*}}(M)$.
\end{thm}

\begin{proof}
Step1. For any $u\in H^{1,2}(M)$ the following holds:
\begin{equation}
\|u\|_{2^{*}}^{2^{*}}\le C\|u\|_{H^{1,2}}^{2}\sup_{j\in2^{\frac{N-2}{2}\Z}}\left(\int_{j\le|u(x)|\le2^{\frac{N-2}{2}}j}|u|^{2^{*}}\mathrm{d}\mu\right)^{\frac{2}{N}}.\label{eq:prange}
\end{equation}
Indeed, let $\chi\in C_{0}^{1}(2^{-\frac{N-2}{2}},2^{N-2})$, extended
by zero to $[0,\infty)$ be such that $\chi(s)\in[0,1]$ for all $s$
and $\chi(s)=1$ whenever $s\in[1,2^{\frac{N-2}{2}}]$. Let $\chi_{j}(s)=j\chi(j^{-1}s)$,
${\normalcolor j\in2^{\frac{N-2}{2}\Z}}$.

Applying Sobolev inequality to $\chi_{j}(|u|)$ we get 

\[
\left(\int_{j\le|u(x)|\le2^{\frac{N-2}{2}}j}|u|^{2^{*}}\mathrm{d}\mu\right)^{2/2^{*}}\le C\int_{2^{-\frac{N-2}{2}}j\le|u(x)|\le2^{N-2}j}(|du|^{2}+|u|^{2})\mathrm{d}\mu,
\]
from which we have
\[
\int_{j\le|u(x)|\le2^{\frac{N-2}{2}}j}|u|^{2^{*}}\mathrm{d} \mu \ \ \le \]
\[C\int_{2^{-\frac{N-2}{2}}j\le|u(x)|\le2^{N-2}j}(|du|^{2}+|u|^{2})\mathrm{d}\mu\left(\int_{j\le|u(x)|\le2^{\frac{{N-2}}{2}}j}\mathrm{|u|^{2^{*}}d}\mu\right)^{1-\frac{2}{2^{*}}}.
\]
Adding the inequalities over $j\in2^{\frac{N-2}{2}\Z}$ , while evaluating
the last term by its upper bound, we get (\ref{eq:prange}). 

Step 2. Apply (\ref{eq:prange}) to the sequence $u_{k}$. 

Let $j_{k}\in2^{\frac{N-2}{2}\Z}$ be such that 
\begin{equation}
\sup_{j\in2^{\frac{N-2}{2}\Z}}\int_{j\le|u_{k}(x)|\le2^{\frac{N-2}{2}}j}|u|_{k}^{2^{*}}\mathrm{d}\mu\le2\int_{j_{k}\le|u_{k}(x)|\le2^{\frac{N-2}{2}}j_{k}}|u|^{2^{*}}\mathrm{d}\mu.\label{eq:nearmax}
\end{equation}
Then we have from (\ref{eq:prange}) 
\begin{equation}
\|u_{k}\|_{2^{*}}\le C\left(\int_{j_{k}\le|u(_{k}x)|\le2^{\frac{N-2}{2}}j_{k}}|u_{k}|^{2^{*}}\mathrm{d}\mu\right)^{\frac{2}{2^{*}N}}\label{eq:redux1}
\end{equation}

It suffices to consider two cases: $j_{k}\le L$ for all $k$, $L\in\Z$,
and $j_{k}\to+\infty$. In the first case we have from (\ref{eq:redux1})
with any small $\varepsilon>0,$
\[
\|u_{k}\|_{2^{*}}\le C\left(\int_{j_{k}\le|u_{k}(x)|\le2^{\frac{N-2}{2}}j_{k}}|u_{k}|^{2^{*}}\mathrm{d}\mu\right)^{\frac{2}{N2^{*}}}\le C\left(L^{2^{*}-p}\int_{M}|u_{k}|^{p}\mathrm{d}\mu\right)^{\frac{2}{N2^{*}}}
\]
which converges to $0$ by the assumption on the sequence. 
%Indeed, since $u_{k}\to0$ in $L^{p}(M)$
%for some $p<2^{*}$, the right hand side above vanishes once one applies
%the Hölder inequality and the fact that $\|u_{k}\|_{2^{*}}$ is bounded. 

Step 3. Consider now the second case, $j_{k}\to\infty$. For each  $k\in\N$, let $Y_{k}\subset M$
be a discretization of $M$ given by Lemma \ref{lem:covering}, that is, a collection of balls $\lbrace B(y,t_{k}r)\rbrace_{y\in Y_{k}}$ with
$t_{k}=j_{k}^{-\frac{2}{N-2}},$ form a covering for $M$ of uniformly
finite multiplicity. Let $D_{k}=\lbrace x\in M:|u_{k}(x)|\in[j_{k},2^{\frac{N-2}{2}}j_{k}]\rbrace$
and $D'_{k}=\lbrace x\in M:|u_{k}(x)|\in[2^{-\frac{N-2}{2}}j_{k},2^{N-2}j_{k}]\rbrace$.

By scaling of the Sobolev inequality, applied to $\chi_{j_{k}}(|u_{k}|)$
on the geodesic ball $B(y,t_{k}r)$ there is a constant $C$ independent
of $k$ such that
\[
\left(\int_{B(y,t_{k}r)\cap D_{k}}|u_{k}|^{2^{*}}\mathrm{d}\mu\right)^{2/2^{*}}\le C\int_{B(y,t_{k}r)\cap D'_{k}}(|du_{k}|^{2}+t_{k}^{-2}|u_{k}|^{2})\mathrm{d}\mu.
\]
Taking into account that the integration domain in the right hand
side is a subset of $D'_{k},$ we have $t_{k}^{-2}|u_{k}|^{2}\le C|u_{k}|^{2^{*}}$
uniformly in $k$ , and thus we have 
$$
\begin{array}{c}
\int_{B(y,t_{k}r)\cap D_{k}}|u_{k}|^{2^{*}}\mathrm{d}\mu \ \le \\ \int_{B(y,t_{k}r)\cap D'_{k}}(|du_{k}|^{2}+|u_{k}|^{2^{*}})\mathrm{d}\mu\left(\int_{B(y,t_{k}r)\cap D_{k}}|u_{k}|^{2^{*}}\mathrm{d}\mu\right)^{\frac{2}{N}}.
\end{array}$$

Adding the inequalities above over $y\in Y_{k}$ and replacing the
second term in the right hand side by its upper bound, we have 
\begin{equation}
\int_{D_{k}}|u_{k}|^{2^{*}}\mathrm{d}\mu\le C\sup_{y\in Y_{k}}\left(\int_{B(y,t_{k}r)\cap D_{k}}|u_{k}|^{2^{*}}\mathrm{d}\mu\right)^{\frac{2}{N}}.\label{eq:redux2}
\end{equation}
Choosing points $y_{k}\in Y_{k}$ so that 
\[
\sup_{y\in Y_{k}}\int_{B(y,t_{k}r)\cap D_{k}}|u_{k}|^{2^{*}}\mathrm{d}\mu\le2\int_{B(y_{k},t_{k}r)\cap D_{k}}|u_{k}|^{2^{*}}\mathrm{d}\mu,
\]
we have from (\ref{eq:redux1}) and (\ref{eq:redux2}), 
\[
\|u_{k}\|_{2^{*}}\le C\left(\int_{B(y_{k},t_{k}r)\cap D_{k}}|u_{k}|^{2^{*}}\mathrm{d}\mu\right)^{\frac{4}{2^{*}N^{2}}}.
\]
Changing variables from a small ball on $M$ to small ball on $\R^{N}$
by the geodesic map $e_{y_{k}}$, and taking into account that $M$ is
a manifold of bounded geometry, so that the estimate of the Jacobian
is uniform in $k$, we have, with $\Delta_{k}=\lbrace\xi\in\Omega_{r}:|u_{k}\circ\exp_{y_{k}}(\xi)|\in(j_{k},2^{\frac{N-2}{2}}j_{k})\rbrace$,
for all $k$ large enough 
\[
\|u_{k}\|_{2^{*}}\le C\left(\int_{\Omega_{t_{k}r}\cap\Delta_{k}}|u_{k}\circ e_{y_{k}}(\xi)|^{2^{*}}\mathrm{d}\xi\right)^{\frac{4}{2^{*}N^{2}}}.
\]
Let us now change variables again by setting $\xi=t_{k}\eta$, $\eta\in\Omega_{r}$:
\[
\|u_{k}\|_{2^{*}}\le C\left(\int_{\lbrace\eta\in\Omega_{r}:|u_{k}\circ e_{y_{k}}(t_{k}\eta)|\in[j_{k},2^{\frac{N-2}{2}}j_{k}]\rbrace}|t_{k}^{\frac{N-2}{2}}u_{k}\circ e_{y_{k}}(t_{k}\eta)|^{2^{*}}\mathrm{d}\eta\right)^{\frac{4}{2^{*}N^{2}}},
\]
and note that the integrand converges to zero almost everywhere by
assumption, and is bounded by the constant $2^{N}$, so by Lebesgue
dominated convergence theorem the right hand side vanishes, which
proves the theorem.
\end{proof}

\section{Profile decomposition for sequences vanishing in $L^{p}$, $p<2^{*}$.}

We will start with a characterization of decoupling of bubbles involved
in our profile decomposition. 
\begin{defn}
We shall say that two sequences, $U_{k}$ and $V_{k}$, in a Hilbert
space $H$ are asymptotically orthogonal, if $(U_{k},V_{k})\to0$.
\end{defn}

Let us fix $\chi\in C_{0}^{\infty}(\Omega_{r})$ such that $\chi(\xi)=1$
whenever $|\xi|\le\frac{r}{2}$, extended by zero to a function on $\R^N$. %
\begin{comment}
We will slightly abuse the definition of weak convergence by saying
that if $(u_{k})$ is a sequence of functions defined on the ball
$\Omega_{r}$,
\end{comment}

\begin{lem}
\label{lem:aorth}Let (with extension to $M\setminus B_{r}(y_{k})$
by zero):
\[
V_{k}(x)=2^{\frac{N-2}{2}j_{k}}\,\chi\circ e_{y_{k}}^{-1}(x)\,v(2^{j_{k}}e_{y_{k}}^{-1}(x)),\;x\in M,
\]
\[
W_{k}(x)=2^{\frac{N-2}{2}\ell_{k}}\,\chi\circ e_{z_{k}}^{-1}(x)\,w(2^{\ell_{k}}e_{z_{k}}^{-1}(x)),\;x\in M
\]
where $j_{k},\ell_{k}\in\N$, $j_{k},\ell_{k}\to\infty$ and $y_{k},z_{k}\in M$.
Then $V_{k}$ and $W_{k}$ are asymptotically orthogonal in $H^{1,2}(M)$
for any choice of $v$ and $w$ in $\dot{H}^{1,2}(\R^{N})$
if and only if the following condition holds:
\begin{equation} \label{eq:orth3}
|\ell_k-j_k| \ + \ (2^{\ell_k}+2^{j_k}) d(y_k,z_k) \rightarrow \infty
\end{equation}
as $k\rightarrow \infty.$

Furthermore, given $w\in\dot H^{1,2}(\R^N)$, $(V_k,W_k)\to 0$ for any $v\in\dot H^{1,2}(\R^N)$ if and only if 
\begin{equation}\label{eq:zlm0}
2^{-j_{k}\frac{N-2}{2}}W_{k}\circ e_{y_{k}}(2^{-j_{k}}\xi)\to 0\mbox{ a.e. in }\R^N .
\end{equation}
\end{lem}
\begin{proof}
1. Note that $\int_{M}V_{k}^{2}d\mu\to0$ and $\int_{M}W_{k}^{2}d\mu\to0$,
so asymptotic orthogonality of $V_{k}$ and $W_{k}$ is equivalent
to $\int_{M}dV_{k}\cdot dW_{k}d\mu\to0$. 
By density of $C_0^\infty(\R^N)$ in $\dot H^{1,2}(\R^N)$ we may assume without loss of generality that $v, w \in C_0^\infty(\R^N)$. 

2. { \it Sufficiency:} First note that supports of $V_k$ and $W_k$ are contained in $B(y_k,r) \cap e_{y_k}(2^{-j_k}\mathrm{supp}\ v)$ and $B(z_k,r) \cap e_{z_k}(2^{-\ell_k}\mathrm{supp}\ w)$ respectively. Thus, if along a subsequence $\inf d(y_k,z_k) >0$, then the supports of $V_{k}$ and $W_{k}$ are disjoint for large $k$ and the conclusion follows. Hence we assume in the rest of the proof that $d(y_k,z_k)\rightarrow 0$ as $k\rightarrow \infty$.\\\\
The support of $dV_{k}\cdot dW_{k}$ is contained in $B(z_k,r)$ and hence calculating the integral in the coordinate chart $e_{z_k}$ we get  
$$ \int_{M}dV_{k}\cdot dW_{k}d\mu = \int_{\Omega_r} g^{\alpha \beta}(e_{z_k}(\xi))\partial_\alpha (V_k\circ e_{z_k})\partial_\beta (W_k\circ e_{z_k}) \sqrt{g(e_{z_k}(\xi))}d\xi$$
where $\partial_1,\dots,\partial_N$ is a shorthand for respective partial derivatives $\frac{\partial}{\partial \xi_1}, \dots , \frac{\partial}{\partial \xi_N}$, 
$\xi_1,\dots,\xi_N$ are the components of $\xi\in \Omega_r$,
 $g_{\alpha\beta}=\langle\frac{\partial}{\partial \xi_{\alpha}},\frac{\partial}{\partial \xi_{\beta}}\rangle_g,\ (g^{\alpha\beta}) = (g_{\alpha\beta})^{-1},\ g = \det (g_{\alpha\beta})$ and we use the summation convension over $\alpha,\beta \in \{1,...,N\}.$\\\\
Denoting  $j_{k}-\ell_{k}=m_{k},\ e_{y_k}^{-1}\circ e_{z_k}\ = \ \psi_k $, using the change of variable $\eta=2^{\ell_{k}}\xi $, and taking into account that $g^{\alpha\beta}(z_k) = \delta_{\alpha\beta} $ and $g(z_k)=1$, the above expression simplifies to $\int_{M}dV_{k}\cdot dW_{k}d\mu = \int_{B(z_k,r)\cap B(y_k,r)}dV_{k}\cdot dW_{k}d\mu = 2^{\frac{N-2}{2}m_k} \times$ 
$$  \int_{D_k}g^{\alpha \beta}(e_{z_k}(\xi))\partial_\alpha (\chi \circ \psi_k (\xi)v(2^{j_{k}}\psi_k(\xi)))\partial_\beta (\chi(\xi) w(2^{\ell_k}\xi)) \sqrt{g(e_{z_k}(\xi))}d\xi, $$
$$= 2^{\frac{N}{2}m_k}\int_{2^{\ell_k}D_k \cap \mathrm{supp}w}\left[\nabla w (\eta) \psi'_k(2^{-\ell_{k}}\eta)\cdot \nabla v (2^{j_k}\psi_k(2^{-\ell_{k}}\eta))  + o(1)\right]d\eta\
$$
where $D_k= e_{z_k}^{-1}\left(B(z_k,r)\cap B(y_k,r)\right)$, $\psi_k'$ is the $N\times N$-matrix derivative of $\psi_k$, and $o(1)$ denotes a sequence converging to zero uniformly in $\R^{N}$. Since $M$ is of bounded geometry using standard arguments we may assume $\psi_k$ and its derivatives will converge locally uniformly in $\Omega_r$ to some $\psi \in C^\infty(\Omega_r \rightarrow \overline{\Omega_r})$ and its respective derivatives.\\
Since \eqref{eq:orth3} holds, we have up to subsequences two cases :\\
{\it Case 1 :}  $|\ell_k-j_k| \rightarrow \infty $.\\
In this case we may assume without loss of generality that $j_{k}-\ell_{k}=m_{k} \rightarrow -\infty $ and hence using the boundednesss of $\psi_k$ and the fact that $v,w \in C_c^\infty(\R^N)$ we get 
$$
\left|\int_{M}dV_{k}\cdot dW_{k}d\mu\right|\  \le \ C 2^{\frac{N}{2}m_k}\  \rightarrow 0.  
$$
{\it Case 2 :} $|\ell_k-j_k|$ is bounded and $ \ (2^{\ell_k}+2^{j_k}) d(y_k,z_k) \rightarrow \infty $.\\
For $\eta \in 2^{\ell_k}D_k \cap \mathrm{supp}\ w$,we have 
$$2^{j_k}\psi_k(2^{-\ell_{k}}\eta)= 2^{j_k}\psi_k(0) + 2^{j_k}\psi'_k(0)2^{-\ell_{k}}\eta \ + \ O(2^{j_k-2\ell_{k}})$$
Note that since the exponential map is preserving distance from the origin in $\Omega_r$, we have $|\psi_k(0)|= d(y_k,z_k)$. Using this with the fact that $m_k$ is bounded we get $|2^{j_k}\psi_k(2^{-\ell_{k}}\eta)| \rightarrow \infty\ ,  \ {\rm as}\ \ k\rightarrow \infty$ (with uniform convergence) and hence $\nabla v (2^{j_k}\psi_k(2^{-\ell_{k}}\eta))=0$
for all $\eta \in \mathrm{supp}\ w$ provided that $k$ is sufficiently large, and hence the conclusion follows.

3. {\it Necessity:} Assume that the condition (\ref{eq:orth3}) does not hold. Then, on a renamed subsequence,  $j_{k}-\ell_{k}=m\in\Z$ and $2^{j_k}d(y_k,z_k)= |2^{j_k}\psi_k(0)|$ remains bounded. Hence passing to a further subsequence if necessary we may assume $2^{j_k}\psi_k(0)\to\eta_0\in\R^N$. Then, repeating calculations in the proof of sufficiency, we get  
$$
	\int_{M}dV_{k}\cdot dW_{k}d\mu\rightarrow 2^{\frac{N}{2}m}\int_{\R^N}\nabla w (\eta) \cdot \psi'(0)\nabla v (\eta_0+2^{m}
	\psi'(0)\eta)\,d\eta$$
	
Using elementary properties of the exponential map we can show that $\psi'(0)$ is invertible. Hence the above expression is nonzero, for example, when $v\neq 0$ and $w(\eta)= v(\eta_0+ 2^m\psi'(0)\eta)$.

4. Finally, writing the scalar product of $H^{1,2}(M)$ in the normal coordinates at $y_k$ we have
\begin{align*}
(W_k,V_k)=\int_{\R^N}\nabla 2^{-j_{k}\frac{N-2}{2}}W_{k}\circ e_{y_{k}}(2^{-j_{k}}\xi)\cdot\nabla v(\xi) d\xi+o(1),
\end{align*}
which proves that \eqref{eq:zlm0} is equivalent to asymptotic orthogonality of $V_k,W_k$ for all $v$. 
\end{proof}
We now provide a profile decomposition for sequences that are bounded
in $H^{1,2}(M)$ and vanish in $L^{p}(M)$ for some $p\in(2,2^{*})$.
This will allow us to consider a general bounded sequence in $H^{1,2}(M)$,
for which a profile decomposition with a remainder vanishing in $L^{p}$
is already known, identify blowups in this remainder, and after subtracting
them, obtain a profile decomposition for the remainder up to a term
that, by Theorem \ref{thm:vanishing}, will vanish already in $L^{2^{*}}$.

\begin{thm}
\label{thm:PD-blowupsonly}Let $M$ be a complete smooth $N$-dimensional
Riemannian manifold of bounded geometry, and let $(u_{k})$ be a bounded
sequence in $H^{1,2}(M)$ convergent to zero in $L^{p}(M)$ for some
$p\in(2,2^{*})$. Then, there exist sequences of points $(y_{k}^{(n)})_{k\in\N}$
in $M$ and of numbers $(j_{k}^{(n)})_{k\in \N}$ in $\N$, $j_{k}^{(n)}\to+\infty$,
as well as functions $w^{(n)}\in\dot{H}^{1,2}(\R^{N})$, $n\in\N,$
such that 
\begin{equation}
2^{-j_{k}^{(n)}\frac{N-2}{2}}u_{k}\circ e_{y_{k}^{(n)}}(2^{-j_{k}^{(n)}}\xi)\to w^{(n)}(\xi)\;\mbox{ a.e. in }\R^{N},\label{eq:profiles01}
\end{equation}
$\mathbf{(AO)}$\qquad{}Whenever $m\neq n$, the sequences $j_{k}^{(m)},$
$y_{k}^{(m)}$, $j_{k}^{(n)},$ $y_{k}^{(n)}$ satisfy the
condition (\ref{eq:orth3}) with $j_{k}=j_{k}^{(m)},$
$y_{k}=y_{k}^{(m)}$, $\ell_{k}=j_{k}^{(n)}$ and $z_k=y_{k}^{(n)}$.

The series $S_{k}=\sum_{n\in\N}W_{k}^{(n)},$ where
\begin{equation}
W_{k}^{(n)}(x)=2^{j_{k}^{(n)}\frac{N-2}{2}}\,\chi\circ e_{y_{k}^{(n)}}^{-1}(x)\,w^{(n)}(2^{j_{k}^{(n)}}e_{y_{k}^{(n)}}^{-1}(x)),\:x\in M,\label{eq:Sksum}
\end{equation}

converges in $H^{1,2}(M)$ unconditionally with respect to $n$ and uniformly in $k$,

\begin{equation}
u_{k}-S_{k}\to0\mbox{ in }L^{2^{*}}(M)\, ,\label{eq:remainder}
\end{equation}

and 

\begin{equation}
\sum_{n\in\N}\int_{\R^{N}}|\nabla w^{(n)}|^{2}dx\le\liminf\int_{M}|du_{k}|^{2}d\mu.\label{eq:Planch01}
\end{equation}
\end{thm}

\begin{proof}
The proof is analogous to the proofs for profile decompositions in
\cite{SkrTi4,SchinTin,Solimini,SoliTi} and is given in a somewhat
abbreviated form.

1. Consider any sequences $(j_{k}^{(1)})$ in $\N$, $j_{k}^{(1)}\to\infty$,
and $(y_{k}^{(1)})$ in $M$. The sequence $2^{-j_{k}^{(1)}\frac{N-2}{2}}u_{k}\circ e_{y_{k}^{(1)}}(2^{-j_{k}^{(1)}}\cdot)$
is a bounded sequence in $H^{1,2}(\Omega_{a})$ for any $a\in(0,r(M))$ and
thus has a weakly convergent subsequence there. Consequently, by diagonalization,
there is a renamed subsequence of $(u_{k})$ such that $2^{-j_{k}^{(1)}\frac{N-2}{2}}u_{k}\circ e_{y_{k}^{(1)}}(2^{-j_{k}^{(1)}}\cdot)$
converges almost everywhere on $\R^{N}$to some $w^{(1)}$. Note that,
since $M$ is of bounded geometry, $\|\nabla w^{(1)}\|_{2}^{2}\le\limsup C\int_{B_{r}(y_{k}^{(1)})}|du_{k}|^{2}d\mu<\infty$ with the constant $C$ independent of the sequence $(y_k^{(n)})_{k\in\N}$.
Thus it suffices to prove the theorem for the case $w^{(1)}\neq 0$. Let $\Xi_1$ be a set of all $w\in\dot H^{1,2}(M)$ such that  
\[
 2^{-j_{k}^{(1)}\frac{N-2}{2}}u_{k}\circ e_{y_{k}^{(1)}}(2^{-j_{k}^{(1)}}\to w \mbox{ a.e. in } \R^N
\]
for some choice of $j_{k}^{(1)}\to\infty$
and $(y_{k}^{(1)})$ in $M$, and let
\[
\beta_1=\sup_{w\in \Xi_1}\|w\|_{\dot{H}^{1,2}}
\]
and fix an element $w^{(1)}\in \Xi$ and corresponding sequences $j_{k}^{(1)}$ and $(y_{k}^{(1)})$ so that $\|w^{(1)}\|_{\dot{H}^{1,2}}\ge \frac12 \beta_1$.
If $\beta_1=0$, then by Theorem~\ref{thm:vanishing} $u_k\to 0$ in $L^{2^*}(M)$ and the theorem is proved. We consider therefore the case $\beta_1>0$. 

2. Let us now make a recursive definition of sequences $j_{k}^{(n)}\to\infty$,  $(y_{k}^{(n)})$ in $M$, functions $w^{(n)}$ and numbers $\beta_n$, $n\in\N$.
Let $\nu\in\N$ and assume that for every $n=1,\dots,\nu$ we have constructed the elementary concentrations for the sequence $(u_k)$, that is, we make the following assumptions.

Assume, for $n=1,\dots,\nu$, that we have a renamed subsequence of $(u_{k})$, sequences $(j_{k}^{(n)})$ in $\N$, $j_{k}^{(n)}\to\infty$, and $(y_{k}^{(n)})$ in $M$. Assume that functions
$w^{(n)}\in\dot{H}^{1,2}(\R^{N})$ satisfy \eqref{eq:profiles01}. Assume that (AO) is satisfied for $m,n=1,\dots,\nu$. Furthermore, let $\Xi_n$ be a set of all $w\in\dot H^{1,2}(M)$ such that  
\[
 2^{-j_{k}^{(n)}\frac{N-2}{2}}u_{k}\circ e_{y_{k}^{(n)}}(2^{-j_{k}^{(n)}})\to w \mbox{ a.e. in } \R^N
\]
for some choice of $j_{k}^{(n)}\to\infty$
and $(y_{k}^{(n)})$ in $M$ satisfying (AO) for $m<n=1,\dots,\nu$, and let
\[
\beta_n=\sup_{w\in \Xi_n}\|w\|_{\dot{H}^{1,2}}
\]
and fix an element $w^{(n)}\in \Xi$ and corresponding sequences $j_{k}^{(n)}$ and $(y_{k}^{(n)})$ so that $\|w^{(n)}\|_{\dot{H}^{1,2}}\ge \frac12 \beta_n$.

We now set $S_{k}^{(\nu)}=\sum_{n=1}^{\nu}W_{k}^{(n)}$ and $v_{k}^{(\nu)}=u_{k}-S_{k}^{(\nu)}$.
Consider sequences $(j_{k}^{(\nu+1)})$ in $\N$, $j_{k}^{(\nu+1)}\to\infty$,
and $(y_{k}^{(\nu+1)})$ in $M$. As in the Step 1, we have a renamed
subsequence of $(u_{k})$ such that $2^{-j_{k}^{(\nu+1)}\frac{N-2}{2}}u_{k}\circ e_{y_{k}^{(\nu+1)}}(2^{-j_{k}^{(\nu+1)}}\cdot)$
converges almost everywhere on $\R^{N}$ to some $w^{(\nu+1)}\in\dot{H}{}^{1,2}(\R^{N})$. As before, we consider the class $\Xi_{\nu+1}$ of all such limits, and fix $w^{(\nu+1)}$ and corresponding  $j_{k}^{(\nu+1)}$ and $(y_{k}^{(\nu+1)})$ so that 
 $\|w^{(\nu+1)}\|_{\dot{H}^{1,2}}\ge \frac12 \beta_{\nu+1}$.

%Repeating the argument 
%in Step 1, we have  $w^{(\nu+1)}\in\dot{H}{}^{1,2}(\R^{N})$. If  $w^{(\nu+1)}=0$ for any choice of  $j_{k}^{(\nu+1)}\to\infty$
%and $(y_{k}^{(\nu+1)})$ in $M$, then by Theorem~\ref{thm:vanishing} $u_{k}-S_{k}^{(\nu)}\to 0$ in $L^{2^*}(M)$ and we may assign $w^{n}=0$ for all $n\ge\nu$

%3. If $\beta_{\nu+1}=0$, then $w^{(\nu+1)}=0$ and by Theorem \ref{thm:vanishing},
%we have $v_k^{(\nu)}\to 0$ in $L^{2^*}(M)$. We than set $w^{(\nu+1)}=0$ 

%we have the assertions of the theorem hold with $S_{k}=S_{k}^{(\nu)}$.
%Otherwise, we assume in what follows that $w^{(\nu)}\neq0$ for all $\nu\in\N$

3. Note that for every $n\le\nu$,
\begin{equation}
2^{-j_{k}^{(n)}\frac{N-2}{2}}W_{k}^{(n)}\circ e_{y_{k}^{(n)}}(2^{-j_{k}^{(n)}}\xi)=\chi(2^{-j_{k}^{(n)}}\xi)w^{(n)}(\xi)\to w^{(n)}(\xi)\mbox { a.e. in }\R^N,
\end{equation}
and for each $n'\le\nu$, $n'\neq n$, by (AO) and Lemma~\ref{lem:aorth},
\begin{equation*}
	%\label{eq:zlm}
2^{-j_{k}^{(n)}\frac{N-2}{2}}W_{k}^{(n')}\circ e_{y_{k}^{(n)}}(2^{-j_{k}^{(n)}}\xi)\to 0\mbox{ a.e. in }\R^N,
\end{equation*}
so for every $n\le\nu$,
 \begin{equation}
  2^{-j_{k}^{(n)}\frac{N-2}{2}}S^{(\nu)}_{k}\circ e_{y_{k}^{(n)}}(2^{-j_{k}^{(n)}}\xi)\to w^{(n)}(\xi) \mbox{ a.e. in } \R^N,
 \end{equation}
and thus 
\begin{align}\nonumber
2^{-j_{k}^{(n)}\frac{N-2}{2}}v_{k}^{(\nu)}\circ e_{y_{k}^{(n)}}(2^{-j_{k}^{(n)}}\xi)=
\\
\label{eq:proj4}
2^{-j_{k}^{(n)}\frac{N-2}{2}}(u_{k}-S_{k}^{(\nu)})\circ e_{y_{k}^{(n)}}(2^{-j_{k}^{(n)}}\xi)
\to 0  \mbox{ a.e. in }\R^N.
\end{align}
% $\ 2^{-j_{k}^{(n)}\frac{N-2}{2}}u_{k}\circ e_{y_{k}^{(n)}}(2^{-j_{k}^{(n)}}\cdot)= $
%\[
%2^{-j_{k}^{(n)}\frac{N-2}{2}}v_{k}^{(\nu)}\circ e_{y_{k}^{(n)}}(2^{-j_{k}^{(n)}}\cdot)+2^{-j_{k}^{(n)}\frac{N-2}{2}}S_{k}^{(\nu)}\circ e_{y_{k}^{(n)}}(2^{-j_{k}^{(n)}}\cdot)\to w^{(n)}\mbox{ pointwise.}
%\]
%From here we conclude that 
%\begin{equation}
%2^{-j_{k}^{(n)}\frac{N-2}{2}}v_{k}^{(\nu)}\circ e_{y_{k}^{(n)}}(2^{-j_{k}^{(n)}}\cdot)\to0.
%\end{equation}

4. Let us show that (AO) is satisfied for $m=1,\dots,\nu$ and $n=\nu+1$
(or vice versa). Once we show this, we will have completed the construction
of $(y_{k}^{(n)})_{k\in\N}$ in $M$ and of numbers $(j_{k}^{(n)})_{k\in N}$
in $\N$, $j_{k}^{(n)}\to+\infty$, as well as functions $w^{(n)}\in\dot{H}^{1,2}(\R^{N})$,
$n\in\N,$ such that, for a renamed subsequence of $(u_{k}),$ condition
(AO) is satisfied for all $n\in\N$. If $w^{(\nu+1)}=0$, then, since $\beta_{\nu+1}=0$ we are free replace $(y_{k}^{(\nu+1)})_{k\in\N}$ in $M$ and $(j_{k}^{(\nu+1)})_{k\in N}$ with any sequence that saisfies \eqref{eq:orth3} for respective scaling sequences, namely
\[
|j_k^{(\nu+1)}-j_k^{(n)}| \ + \ (2^{j_k^{\nu+1}}+2^{j_k^{(n)}}) d(y_k^{(\nu+1)},y_k^{(n)}) \rightarrow \infty, n=1,\dots\nu.
\]
The renamed  $w^{(\nu+1)}$ will be necessarily zero since $\beta_{\nu+1}=0$.

We now may assume that  $w^{(\nu+1)}\neq 0$.
Assume that (AO) does not hold
for indices $\nu+1$ and some $\ell\le\nu$. Then there exist $m\in\Z$ and $\lambda\in\R$,
such that, on a renamed subsequence, $2^{j_{k}^{(\ell)}}d(y_{k}^{(\ell)},y_{k}^{(\nu+1)})$ is bounded
and $j_{k}^{(\nu+1)}=j_{k}^{(\ell)}-m$. Let  $\psi_k=e_{y_{k}^{(\ell)}}^{-1}\circ e_{y_{k}^{(\nu+1)}}$. Note that $d(y_{k}^{(\ell)},y_{k}^{(\nu+1)})\to 0$, and since $M$ has bounded geometry, on a renamed subsequence we have $\psi_k$ convergent uniformly, together with its derivatives of every order, and its limit is the identity map. Also on a renamed subsequence we have an $\eta_0\in \R^N$ such that
%and a constant $N\times N$-matrix $\Psi$ such that 
$2^{j_{k}^{(\nu+1)}}\psi_k(0)\to \eta_0$ (since $|\psi_k(0)|=d(y_{k}^{(\nu+1)},y_{k}^{(\ell)})$). 
%and $\psi'_k(0)\to \Psi$.
Then we have, uniformly on compact subsets of $\R^N$, 
\begin{align*}
	2^{j_{k}^{(\ell)}}\psi_k(2^{-j_{k}^{(\nu+1)}}\xi)=
	\\
2^{m}2^{j_{k}^{(\nu+1)}}\psi_k(0)+2^{m}\psi_k'(0)\xi+2^{j_{k}^{(\ell)}}O(2^{-2j_{k}^{(\nu+1)}}\xi)\to 2^{m}\eta_0+2^m\xi.
\end{align*}
Note also that for any $n=\ell+1,\dots,\nu$ from (AO) and Lemma~\ref{lem:aorth} we have 
\[
2^{j_{k}^{(\nu+1)}\frac{N-2}{2}}W_k^{(n)}\circ e_{y_{k}^{(\nu+1)}}(2^{-2j_{k}^{(\nu+1)}}\xi)\to 0.
\]

Then, substituting the two last calculations into the expression below, we have
\begin{align*}
	2^{j_{k}^{(\nu+1)}\frac{N-2}{2}}v_k^{(\nu)}\circ e_{y_{k}^{(\nu+1)}}(2^{-2j_{k}^{(\nu+1)}}\xi)=
	\\
	2^{j_{k}^{(\nu+1)}\frac{N-2}{2}}v_k^{(\ell)}\circ e_{y_{k}^{(\nu+1)}}(2^{-2j_{k}^{(\nu+1)}}\xi)+o(1)=
	\\
	2^{-m\frac{N-2}{2}}
	2^{j_{k}^{(\ell)}\frac{N-2}{2}}v_k^{(\ell)}\circ e_{y_{k}^{(\ell)}}\circ \psi_k(2^{-2j_{k}^{(\nu+1)}}\xi)=\\
	2^{-m\frac{N-2}{2}}
	2^{j_{k}^{(\ell)}\frac{N-2}{2}}v_k^{(\ell)}\circ e_{y_{k}^{(\ell)}}(2^{j_{k}^{(-\ell)}} 
	[2^{j_{k}^{(\ell)}}\psi_k(2^{-2j_{k}^{(\nu+1)}}\xi)])=
	\\
	2^{-m\frac{N-2}{2}}
	2^{j_{k}^{(\ell)}\frac{N-2}{2}}v_k^{(\ell)}\circ e_{y_{k}^{(\ell)}}(2^{j_{k}^{(-\ell)}} 
	[2^m\eta_0+2^m\xi+o(1)])	
	\\
	\to 0,
	\end{align*}
by \eqref{eq:proj4}, 
which by definition of $w^{(\nu+1)}$ implies $w^{(\nu+1)}=0$, a
contradiction.\\

6. Let us expand by bilinearity a trivial inequality \\$\int_{M}|d(u_{k}-S_{k}^{(\nu)})|^{2}d\mu\ge0$.
We will have
\begin{equation}
\int_{M}|du_{k}|^{2}d\mu\ge2\int_{M}du_{k}\cdot dS_{k}^{(\nu)}d\mu-I_{k}-I'_{k},\label{eq:trivpos}
\end{equation}

where 
\[
I_{k}=\sum_{n=1}^{\nu}\int_{M}\left|d\left(2^{j_{k}^{(n)}\frac{N-2}{2}}\,\chi\circ e_{y_{k}^{(n)}}^{-1}(x)\,w^{(n)}(2^{j_{k}^{(n)}}e_{y_{k}^{(n)}}^{-1}(x))\right)\right|^{2}d\mu,
\]
and 
\[
I'_{k}=\sum_{m\neq n,\,m,n=1,\dots\nu}\int_{M}d\left(2^{j_{k}^{(n)}\frac{N-2}{2}}\,\chi\circ e_{y_{k}^{(n)}}^{-1}(x)\,w^{(n)}(2^{j_{k}^{(n)}}e_{y_{k}^{(n)}}^{-1}(x))\right)
\cdot\]
\[d\left(2^{j_{k}^{(m)}\frac{N-2}{2}}\,\chi\circ e_{y_{k}^{(m)}}^{-1}(x)\,w^{(m)}(2^{j_{k}^{(m)}}e_{y_{k}^{(m)}}^{-1}(x))\right)d\mu.
\]
The first term in (\ref{eq:trivpos}) can be evaluated by writing
the integration in rescaled geodesic coordinates $\xi=2^{j_{k}^{(n)}}e_{y_{k}^{(n)}}^{-1}(x)$
and denoting by %
\begin{comment}
$o_{w}(1)$ any term that converges to zero weakly in $\dot{H}^{1,2}(\R^{N})$
and by 
\end{comment}
$o(1)$ any term vanishing uniformly in $\R^{N}$ as $k\to\infty$, we 
$$\begin{array}{l}
 \int_{M}du_{k}\cdot dS_{k}^{(\nu)}d\mu=
\\
\sum_{n=1}^{\nu}2^{j_{k}^{(n)}\frac{N-2}{2}}\,\int_{B_{r}(y_{k}^{(n)})}du_{k}\cdot d\left(\chi\circ e_{y_{k}^{(n)}}^{-1}(x)\,w^{(n)}(2^{j_{k}^{(n)}}e_{y_{k}^{(n)}}^{-1}(x))\right)d\mu
\\
=\sum_{n=1}^{\nu}2^{-j_{k}^{(n)}\frac{N-2}{2}}\int_{\Omega_{2^{j_{k}^{(n)}}r}}\sum_{\alpha,\beta=1}^{N}(\delta_{\alpha\beta}+o(1))\partial_{\alpha}u_{k}\circ e_{y_{k}^{(n)}}(2^{-j_{k}^{(n)}}\xi)\times
\\
\left(\chi(2^{-j_{k}^{(n)}}\xi)\partial_{\beta}w^{(n)}(\xi)+o(1)
\partial_{\beta}w^{(n)}(\xi)\right)d\xi\\
=\sum_{n=1}^{\nu}\int_{\R^{N}}|\nabla w^{(n)}(\xi)|^{2}d\xi \ + \  o(1)
\end{array}$$
\\
An analogous evaluation of $I_{k}$ gives, $I_{k}=$
$$
\begin{array}{l}
\sum_{n=1}^{\nu}2^{-j_{k}^{(n)}\frac{N-2}{2}}\int_{\Omega_{2^{j_{k}^{(n)}}r}}\sum_{\alpha,\beta=1}^{N}
\left(\chi(2^{-j_{k}^{(n)}}\xi)\partial_{\alpha}w^{(n)}(\xi)+o(1)\partial_{\alpha}w^{(n)}(\xi)\right)\times
\\
\left(\chi(2^{-j_{k}^{(n)}}\xi)\partial_{\beta}w^{(n)}(\xi)+o(1)\partial_{\beta}w^{(n)}(\xi)\right)(\delta_{\alpha\beta}+o(1))d\xi \\
= \sum_{n=1}^{\nu}\int_{\R^{N}}|\nabla w^{(n)}(\xi)|^{2}d\xi \ + \ o(1).
\end{array}$$
By (AO) and Lemma \ref{lem:aorth}, we have $I'_{k}\to0$. Consequently,
(\ref{eq:trivpos}) implies
\begin{align*}
\int_{M}|du_{k}|^{2}d\mu\ge2\sum_{n=1}^{\nu}\int_{\R^{N}}|\nabla w^{(n)}(\xi)|^{2}d\xi-\sum_{n=1}^{\nu}\int_{\R^{N}}|\nabla w^{(n)}(\xi)|^{2}d\xi+o(1)
\\
=\sum_{n=1}^{\nu}\int_{\R^{N}}|\nabla w^{(n)}(\xi)|^{2}d\xi+o(1).\label{eq:prePl}
\end{align*}

Since $\nu$ is arbitrary, we have (\ref{eq:Planch01}). 

7. Inequality (\ref{eq:Planch01}) implies that $\beta_{\nu}\to0$
as $\nu\to\infty$. Then, following the argument of \cite{SoliTi} with only trivial modifications,
one can show that, for a suitably renamed sequence, the series $S_{k}$
converges in $H^{1,2}(M)$ unconditionally in $n$ and uniformly in
$k$. 

8. Let us finally show that the sequence $(S_k)$  gives indeed the defect of compactness. Let $j_{k}\in\N$, $j_{k}\to\infty$ and let $y_{k}\in M$.
It suffices to consider two cases.

Case A. For every $n\in\N$, $j_{k}^{(n)}, j_k$ and $y_{k}^{(n)},y_{k}$
satisfy the condition (\ref{eq:orth3}).
Then $2^{-j_{k}\frac{N-2}{2}}S_{k}\circ e_{y_{k}}(2^{-j_{k}}\cdot)\to0$
a.e., as this is true for each term in the series of $S_{k}$ by Lemma
\ref{lem:aorth}, and the convergence in the series for $S_{k}$
is uniform. On the other hand, if there is a renamed subsequence such
that $2^{-j_{k}\frac{N-2}{2}}u_{k}\circ e_{y_{k}}(2^{-j_{k}}\cdot)\to w$
a.e., we will have necessarily $\|\nabla w\|_{2}\le\beta_{\nu}$ for
every $\nu\in\N$, i.e. $w=0$. We conclude that $2^{-j_{k}\frac{N-2}{2}}(u_{k}-S_{k})\circ e_{y_{k}}(2^{-j_{k}}\cdot)\to0$
a.e. in this case.

Case B. Without loss of generality we may assume that for some $\ell\in\N$ $j_{k}-j_{k}^{(\ell)}=m\in\Z$ and that $2^{j_k}d(y_{k}^{(\ell)},y_{k})$  is bounded. Then, repeating the argument of Step 4, we have 
\[
2^{-j_{k}\frac{N-2}{2}}(u_{k}-S_{k}^{(\ell)}\circ e_{y_{k}}(2^{-j_{k}}\cdot))\to0\mbox{ a.e. },
\]
while by (AO), Lemma~\ref{lem:aorth} and the uniform convergence of the series
$S_{k}$, 
\[
2^{-j_{k}\frac{N-2}{2}}((S_{k}-S_{k}^{(\ell)})\circ e_{y_{k}}(2^{-j_{k}}\cdot))\to0\mbox{ a.e.},
\]
which implies that $2^{-j_{k}\frac{N-2}{2}}(u_{k}-S_{k})\circ e_{y_{k}}(2^{-j_{k}}\cdot)\to0$
a.e. in this case as well.

Consequently, by Theorem \ref{thm:vanishing} we have $u_{k}-S_{k}\to0$
in $L^{2^{*}}.$ 
\end{proof}

We supplement Theorem \ref{thm:PD-blowupsonly} with the estimate
of $L^{2^{*}}$-norms.
\begin{prop}
\label{prop:BL*}Let $u_{k}$ and $w^{(n)}$, $n\in\N$, be provided
by Theorem \ref{thm:PD-blowupsonly}. Then
\[
\int_{M}|u_{k}|^{2^{*}}d\mu\to\sum_{n\in\N}\int_{\R^{N}}|w^{(n)}|^{2^{*}}dx.
\]
\end{prop}

\begin{proof}
By (\ref{eq:remainder}) it suffices to show that
\[
\int_{M}|\sum_{n\in\N}W_{k}^{(n)}|^{2^{*}}d\mu\to\sum_{n\in\N}\int_{\R^{N}}|w^{(n)}|^{2^{*}}dx.
\]

Since sequences $(W_{k}^{(n)})_{k\in\N}$ have asymptotically disjoint
supports in the sense of (\ref{eq:orth3}), and since one can by density of $C_c(\R^N)$ in $\dot H^{1,2}(\R^N)$ assume that every profile $w^{(n)}$ has compact support, it is easy to show that
\[
\int_{M}|\sum_{n\in\N}W_{k}^{(n)}|^{2^{*}}d\mu-\sum_{n\in\N}\int_{M}|W_{k}^{(n)}|^{2^{*}}d\mu\to0,
\]
so it suffices to show that for each $n\in\N$
\[
\int_{M}|W_{k}^{(n)}|^{2^{*}}d\mu\to\int_{\R^{N}}|w^{(n)}|^{2^{*}}dx.
\]
Indeed, passing to geodesic coordinates at $y_{k}^{(n)}$, and subsequently
rescaling them, we have
\[
\int_{M}|W_{k}^{(n)}|^{2^{*}}d\mu=2^{j_{k}^{(n)}N}\int_{M}\,|\chi\circ e_{y_{k}^{(n)}}^{-1}(x)\,w^{(n)}(2^{j_{k}^{(n)}}e_{y_{k}^{(n)}}^{-1}(x))|^{2^{*}}d\mu=
\]
\[
2^{j_{k}^{(n)}N}\int_{\Omega_{\rho}}\,|\chi(\xi)\,w^{(n)}(2^{j_{k}^{(n)}}\xi)|^{2^{*}}\sqrt{g(\xi)}d\xi=\]
\[\int_{_{\R^{N}}}|\chi(2^{-j_{k}^{(n)}}\eta)\,w^{(n)}(\eta)|^{2^{*}}\sqrt{g(2^{-j_{k}^{(n)}}\eta)}d\eta\to\int_{_{\R^{N}}}|w^{(n)}(\eta)|^{2^{*}}d\eta.
\]
Passing to the limit at the last step is justified by Lebesgue dominated
convergence theorem, and by the fact that in normal coordinates $g(0)=1$. 
\end{proof}

\section{Profile decomposition for $H^{1,2}(M)\protect\hookrightarrow L^{2^{*}}(M)$}

\subsection{Profile decomposition for $H^{1,2}(M)\protect\hookrightarrow L^{p}(M)$,
subcritical case}

In order to quote the profile decomposition for the subcritical Sobolev
embedding from \cite{SkrTi4}, we need to refer to Appendix for definitions of the objects involved there. 
\vskip2mm 
\noindent
1. Discretization $Y_{\rho}$ of a metric space $M$, $\rho>0$, is an
at most countable subset of $M$ such that 
\begin{equation}
d(y,y')\ge\rho/2\mbox{ whenever }y\neq y',y,y'\in Y_{\rho}\label{eq:discrete-1}
\end{equation}
 the collection of balls $\lbrace B(y,\rho)\rbrace_{y\in Y_{\rho}}$
is a covering of $M$, and for any $a\ge\rho$ the multiplicity of
the covering $\lbrace B(y,a)\rbrace_{y\in Y}$ is uniformly finite.
Manifolds of bounded geometry always admit a discretization for sufficiently
small $\rho$, see Appendix.

\noindent
2. Manifold at infinity $M_{\infty}^{(y_{k})}$, of a manifold $M$ with
finite geometry, associated with a sequence $(y_{k})$ from a discretization
$Y_{\rho}$ of $M$ is a Riemannian manifold defined in the Appendix.
Like $M$ itself, $M_{\infty}^{(y_{k})}$ has an atlas $\lbrace\varphi_{i}^{(y_{k})},\Omega_{\rho}\rbrace_{i\in\N}$
with identical coordinate neighborhoods $\Omega_{\rho}$. 

\noindent 
3. Notation $M_{\infty}^{(y_{k})}$ is somewhat abbreviated as the construction of the manifold at infinity is in fact dependent on the
choice of a trailing system $\lbrace(y_{k,i})_{k\in\N}\rbrace_{i\in\N}$
of the sequence $(y_{k})$. For each $k\in\N$, $\lbrace y_{k,i}\rbrace_{i\in\N}$
is an ordering (not necessarily unique) of $Y_{\rho}$ by distance
from the point $y_{k}$.
\begin{thm}
\label{thm:pd-subcr}Let $M$ be a manifold of bounded geometry with
a discretization $Y_{\rho}$, $\rho<r(M)/8$, and let $(u_{k})$ be
a sequence in $H^{1,2}(M)$ weakly convergent to some function $u$
in $H^{1,2}(M)$. Then there exists a renamed subsequence of $(u_{k})$,
sequences $(y_{k}^{(n)})_{k\in\N}$ on $Y_{\rho}$, and associated
with them global profiles $w^{(n)}$ on the respective manifolds at
infinity $M_{\infty}^{(n)}\eqdef M_{\infty}^{(y_{k}^{(n)})}$, $n\in\N$,
such that $d(y_{k}^{(n)},y_{k}^{(m)})\to\infty$ when $n\neq m$,
\begin{equation}
u_{k}-u-\sum_{n\in\N}W_{k}^{(n)}\to0\mbox{ in }L^{p}(M),\;p\in(2,2^{*}),\label{eq:PD}
\end{equation}
where 
\[
W_{k}^{(n)}=\sum_{i\in\N_{0}}\chi\circ e_{y_{k,i}^{(n)}}^{-1}w_{i}^{(n)}\circ e_{y_{k,i}^{(n)}}^{-1},
\]	
\[
w_{i}^{(n)}=w^{(n)}\circ\varphi_{i}^{(n)}=\wlim u_k\circ e_{y_{k,i}}^{(n)},
\] 
$\lbrace\varphi_{i}^{(n)},\Omega_{\rho}\rbrace_{i\in\N}$
is the atlas of $M_{\infty}^{(n)}$, and the series $\sum_{n\in\N}W_{k}^{(n)}$
converges in $H^{1,2}(M)$ unconditionally (with respect to $n$) and uniformly in $k\in\N$.
Moreover, 
\[
\|u\|_{H^{1,2}(M)}^{2}+\sum_{n=1}^{\infty}\|w^{(n)}\|_{H^{1,2}(M_{\infty}^{(n)})}^{2}+\|u_{k}-u-\sum_{n\in\N}W_{k}^{(n)}\|_{H^{1,2}(M)}^{2}\]
\begin{equation}
\le\limsup\|u_{k}\|_{H^{1,2}(M)}^{2}\ ,\label{eq:Plancherel}
\end{equation}
and 
\begin{equation}
\int_{M}|u_{k}|^{p}d\mu\to\int_{M}|u|^{p}d\mu+\sum_{n=1}^{\infty}\int_{M_{\infty}^{(n)}}|w^{(n)}|^{p}d\mu_{\infty}^{(n)}.\label{eq:newBL}
\end{equation}
\end{thm}
\begin{remark}
Note that the original version of (\ref{eq:Plancherel})
\begin{equation}
\|u\|_{H^{1,2}(M)}^{2}+\sum_{n=1}^{\infty}\|w^{(n)}\|_{H^{1,2}(M_{\infty}^{(n)})}^{2}\le\limsup\|u_{k}\|_{H^{1,2}(M)}^{2}\label{eq:oldPln}
\end{equation}
 in \cite{SkrTi4} omits the last term in the left hand side of (\ref{eq:Plancherel}). The inequality in the present
form can be derived, however, from its counterpart in \cite{SkrTi4}
as follows. Let $U_{k}=u+\sum_{n\in\N}W_{k}^{(n)}$. Once we show
that $(u_{k}-U_{k},U_{k})_{H^{1,2}(M)}\to0$, we have (\ref{eq:Plancherel})
by expanding 
\[
0\le\|u_{k}\|{}_{H^{1,2}(M)}^{2}=\|(u_{k}-U_{k})+U_{k}\|{}_{H^{1,2}(M)}^{2}
\]
by bilinearity and applying (\ref{eq:oldPln}) to the sequence $(U_{k})$.
By the uniform convergence of the series in $U_{k}$, it suffices
to show that $(u_{k}-U_{k},u)_{H^{1,2}(M)}\to0$ and that $(u_{k}-U_{k},W_{k}^{(n)})_{H^{1,2}(M)}\to0$.
Since $u_{k}\rightharpoonup u$ and $W_{k}^{(n)}\rightharpoonup0$,
$u_{k}-U_{k}\rightharpoonup0$, the first expression indeed vanishes.
For vanishing of the second expression it suffices to show that 
$(u_k,W_{k}^{(n)})\to\|w^{(n)}\|^2_{H^{1,2}(M_\infty^{(n)})}$ (\cite[relation (5.9)]{SkrTi4}),
$(W_{k}^{(m)},W_{k}^{(n)})_{H^{1,2}(M)}\to 0$ whenever $m\neq n$ (\cite[proof of Lemma~5.5 below (5.19)]{SkrTi4}), and\\
$(W_{k}^{(n)},W_{k}^{(n)})_{H^{1,2}(M)}\to\|w^{(n)}\|^2_{H^{1,2}(M_\infty^{(n)})}$ (\cite[relation (5.14)]{SkrTi4}).
\end{remark}

\subsection{Main result and corollaries}
We will now assume that the parameter $r\in(0,r(M))$ involved in the statements of Sections 2 and 3 satisfies the constraint of Section 4 as well, $r<r(M)/8$. 

\begin{thm}
\label{thm:finalPD}Let $M$ be a Riemannian manifold of bounded geometry
and let $u_{k}$ be a bounded sequence in $H^{1,2}(M)$ weakly convergent
to some $u$.Then there is a renamed subsequence of $u_{k}$,
sequences $(\bar y_k^{(m)})_{k\in\N}$, $m\in\N$, and $(y_k^{*(n)})_{k\in\N}$,  $n\in\N$, of points in $M$, 
sequences $(j_k^{(n)})_{k\in\N}$,  $n\in\N$ of integers,  $j_{k}^{(n)}\to+\infty$ as $k\to\infty$, such that
\begin{equation}
u_{k}-u-\sum_{m\in\N}\bar{W}_{k}^{(m)}-\sum_{n\in\N}W_{k}^{*(n)}\to0\mbox{ in }L^{2^{*}}(M),\label{eq:finalremainder}
\end{equation}
where $\bar{W}_{k}^{(m)}$ are given by Theorem \ref{thm:pd-subcr} (relative to sequences $(\bar y_k^{(m)})_{k\in\N}$),
and 
\begin{equation}
W_{k}^{*(n)}(x)=2^{j_{k}^{(n)}\frac{N-2}{2}}\,\chi\circ e_{y_{k}^{*(n)}}^{-1}(x)\,w^{*(n)}(2^{j_{k}^{(n)}}e_{y_{k}^{*(n)}}^{-1}(x)),\:x\in M,\label{eq:Sksum*}
\end{equation}
where 
\begin{equation}
2^{-j_{k}^{(n)}\frac{N-2}{2}}u_{k}\circ e_{y_{k}^{*(n)}}(2^{-j_{k}^{(n)}}\cdot)\to w^{*(n)}\;\mbox{ a.e. in }\R^{N},\label{eq:profiles01*}
\end{equation}
(as in Theorem~\ref{thm:PD-blowupsonly}), 
and both series in \eqref{eq:finalremainder} converge unconditionally and uniformly in $k$.

Furthermore,
 sequences $j_{k}^{(n)}$,
$y_{k}^{*(n)}$, $j_{k}^{(n')},$ $y_{k}^{*(n')}$ satisfy the
condition (\ref{eq:orth3}).

Moreover, with $M_\infty^{(m)}$, $m\in\N$, as in Theorem~\ref{thm:pd-subcr}, we have

\[\sum_{n\in\N}\int_{\R^{N}}|\nabla w^{*(n)}|^{2}dx+\sum_{m\in\N}\int_{M_{\infty}^{(m)}}(|d\bar{w}^{(m)}|^{2}+|\bar{w}^{(m)}|^{2})d\mu_{\infty}^{(m)} 
\]
\begin{equation}
+ \int_{M}(|du|^{2}+u^{2})d\mu \le\int_{M}(|du_{k}|^{2}+u_{k}^{2})d\mu+o(1),\label{eq:finalPlancherel}
\end{equation}

and

\begin{equation}
\int_{M}|u_{k}|^{2^{*}}d\mu\to\sum_{n\in\N}\int_{\R^{N}}|w^{*(n)}|^{2^{*}}dx+\sum_{m\in\N}\int_{M_{\infty}^{(m)}}|\bar{w}^{(m)}|^{2^{*}}d\mu_{\infty}^{(m)}+\int_{M}|u|^{2^{*}}d\mu.\label{eq:finalBL}
\end{equation}
\end{thm}

\begin{proof}
Apply Theorem \ref{thm:pd-subcr} to $u_{k}$ and let $v_{k}=u_{k}-u-\sum_{m\in\N}\bar{W}_{k}^{(m)}$
be the left hand side of in (\ref{eq:PD}) Note that $v_{k}$ is a
bounded sequence in $H^{1,2}(M)$ because so are both $u_{k}$ and
$\sum_{m\in\N}\bar{W}_{k}^{(m)}$ (for the latter it can be inferred
from (\ref{eq:Plancherel})). Apply Theorem \ref{thm:PD-blowupsonly}
to $v_{k}.$ Then (\ref{eq:finalremainder}) is immediate from combining
(\ref{eq:PD}) and (\ref{eq:remainder}). Relation (\ref{eq:finalPlancherel})
follows from substitution of (\ref{eq:Planch01}) for $v_{k}$ into
(\ref{eq:Plancherel}).

By Brezis-Lieb Lemma, we have
\begin{equation}
\int_{M}|u_{k}|^{2^{*}}d\mu=\int_{M}|u-u-\sum_{m\in\N}\bar{W}_{k}^{(m)}|^{2^{*}}d\mu+\int_{M}|u+\sum_{m\in\N}\bar{W}_{k}^{(m)}|^{2^{*}}+o(1).\label{eq:tempBL}
\end{equation}
Then (\ref{eq:finalBL}) follows by evaluating the first
term in the right hand side of (\ref{eq:tempBL}) by Proposition \ref{prop:BL*}
and evaluating the second term as $\int_{M}|u|^{2^{*}}d\mu+\sum_{m=1}^{\infty}\int_{M_{\infty}^{(m)}}|\bar{w}^{(m)}|^{2^{*}}d\mu_{\infty}^{(n)}$.
The latter has the same form as (\ref{eq:newBL}), but for $p=2^{*}$, 
and can be obtained by literally repeating the argument in \cite{SkrTi4}
applied to the sequence $(u+\sum_{m\in\N}\bar{W}_{k}^{(m)})_{k}$. 
\end{proof}
\begin{cor}
Let $M$ be a manifold of negative curvature with bounded geometry (in particular a hyperbolic space). Let $u_{k}\in H^{1,2}(M)$
satisfy $\int_{M}|du_{k}|^{2}d\mu\le C. $Then $u_{k}$ has a renamed
subsequence satisfying the assertions of Theorem \ref{thm:finalPD}.
\end{cor}
\begin{remark}
Let $M$ be a non-compact symmetric space ( in particular, $\R^N$ or the hyperbolic space $\mathbb H^N$ with $N>2$). Then Theorem~\ref{thm:finalPD} holds with  $M_\infty^{(m)}=M$ for every $m\in\N$, and with $\bar W_k^{(m)}=\bar w^{(m)}\circ\eta_k^{(m)}$, where $w^{(m)}=\wlim u_k\circ{\eta_k^{(m)}}^{-1}$, $\eta_k^{(m)}$ are discrete sequences of isometries on $M$, and the sequences ${\eta_k^{(m)}}^{-1}\circ {\eta_k^{(m')}}$ are discrete whenever $m\neq m'$. This is immediate from the corresponding simplification of Theorem~\ref{thm:pd-subcr} for the case homogeneous spaces, \cite[Theorem~1.1]{SkrTi4}. The same holds also when $M$ coincides with homogeneous space outside of a compact subset.    
\end{remark}

%\begin{rem}
%Note that the dilation profiles $w^{*(n)}$ that define bubbles $W^{*(n)}$
%in Theorem \ref{thm:finalPD} are defined by the weak limiting procedure
%applied to $u_{k}-u-\sum_{m\in\N}\bar{W}_{k}^{(m)}$ rather to $u_{k}$,
%i.e. as the a.e. limit of $2^{-j_{k}^{(n)}\frac{N-2}{2}}(u_{k}-u-\sum_{m\in\N}\bar{W}_{k}^{(m)})\circ e_{y_{k}^{(n)}}(2^{-j_{k}^{(n)}}\xi).$
%It is easy to see, however, that the contribution of $u$ and $\sum_{m\in\N}\bar{W}_{k}^{(m)}$
%to the value of the limit is null and that $w^{*(n)}$ is therefore
%the a.e. limit of $2^{-j_{k}^{(n)}\frac{N-2}{2}}u_{k}\circ e_{y_{k}^{(n)}}(2^{-j_{k}^{(n)}}\xi)$.
%\end{rem}

\section{Appendix}

\subsection{Manifolds of bounded geometry and covering lemma}

In this subsection we list some elementary properties of manifolds
of bounded geometry. 

The following lemma is immediate from \cite[Theorem A ff.]{Eich}
.
\begin{lem}
\label{lem:2.1}Let $p\in(0,\infty)$ and $a\in(0,r(M))$. There exists
a constant $C>1$ such that for any $x\in M$\textup{ and any $u\in H^{1,2}(M)$,}

\begin{equation}
C^{-1}\int_{B_{a}(x)}|u|^{p}d\mu \ \le \ \int_{\Omega_{a}}|u\circ e_{x}\text{\ensuremath{|^{p}}dx \ \ensuremath{\le} \ C\ensuremath{\int}}_{B_{a}(x)}|u|^{p}d\mu,\label{eq:peq}
\end{equation}
and 

\textup{
\[
C^{-1}\int_{B_{a}(x)}|du|^{2}d\mu \ \le \ \int_{\Omega_{a}}\text{|\ensuremath{\nabla}(u\ensuremath{\circ e_{x}})\ensuremath{|^{2}}dx \ \ensuremath{\le} \ C\ensuremath{\int}}_{B_{a}(x)}|du|^{2}d\mu
\]
}
\[
\]
\end{lem}

For any two points $x,y\in M$ define 
\[
\Omega_{a}(x,y)=e_{x}^{-1}(e_{x}\Omega_{a}\cap e_{y}\Omega_{a})\subset\Omega_{a},
\]
and a diffeomorphism 
\[
\psi_{yx}=e_{y}^{-1}\circ e_{x}:\Omega_{a}(x,y)\to\Omega_{a}(y,x).
\]
\begin{lem}
\label{lem:bdd-derivatives}If the manifold $M$ has bounded geometry, 
then for any $\alpha\in\N_{0}^{N}$ there exists a constant $C_{\alpha}>0$,
such that 
\[
|D^{\alpha}\psi_{yx}(\xi)|\le C_{\alpha}\mbox{ whenever }x,y\in M,\text{\;\ensuremath{\Omega_{a}}(x,y)\ensuremath{\neq\emptyset},}\;\xi\in\Omega_{a}(x,y).
\]

whenever $\Omega_{a}(x,y)$ is nonempty.
\end{lem}

This lemma is immediate from Definition \ref{def:bg}.

The following lemma is found, in particular, in \cite[ Lemma 1.1]{Hebey}.
\begin{lem}
\label{lem:covering}Let $M$ be an $N$-dimensional connected Riemannian
manifold with Ricci curvature uniformly bounded from below . Let $\rho>0$. There exists an at
most countable set $Y\in M$ such that 

\begin{equation}
d(y,y')\ge\rho/2\mbox{ whenever }y\neq y',y,y'\in Y,\label{eq:discrete}
\end{equation}
 the collection of balls $\lbrace B(y,\rho)\rbrace_{y\in Y}$ is a
covering of $M$, and for any $a\ge\rho$ the multiplicity of the
covering $\lbrace B(y,a)\rbrace_{y\in Y}$ is uniformly finite.
\end{lem}

\subsection{Manifolds at infinity}

In this subsection we give a cursive summary of the construction of
manifolds at infinity for a Riemannian manifold of bounded geometry,
following \cite{SkrTi4}. We consider a radius $a\in(1,r(M))$
and the discretization $Y=Y_{a}$ of $M$ fixed assured by the previous lemma with $\rho=a$. We will use the notation
$\N_{0}\eqdef\N\cup\lbrace0\rbrace$.
\begin{defn}
\label{def:trailing} Let $(y_{k})_{k\in\N}$ be a sequence in $Y$.
A countable family $\lbrace(y_{k;i})_{k\in\N}\rbrace_{i\in\N_{0}}$
of sequences on $Y$ is called a family of trailing sequences for
$(y_{k})_{k\in\N}$ if for every $k\in\N$ $(y_{k;i})_{i\in\N_{0}}$
is an ordering of $Y$ by the distance from $y_{k}$, that is, an
enumeration of $Y$ such that $d(y_{k,i},y_{k})\le d(y_{k,i+1},y_{k})$
for all $i\in\N_{0}$. In particular, $y_{k,0}=y_{k}$. 
\end{defn}

The trailing family is generally not uniquely defined when for some
$k\in\N$ there are several points of $Y$ with the same distance
from $y_{k}$, so strictly speaking the manifold at infinity is determined
in the construction below not by the sequence $(y_{k})$ but by its
trailing family.

To each pair $(i,j)\in\N_{0}^{2}$ we associate a subset $\Omega_{ij}$
of $\Omega_{\rho}$, where $\rho=\frac{a}{5}$. Note that for each
$i,j\in\N_{0}$, such that $\liminf_{k\to\infty}d(y_{k,i},y_{k,j})\le\rho$,
and with $\xi,\eta\in\Omega_{\rho}$, 
\[
d(e_{y_{k,j}}\xi,e_{y_{k,i}}\eta)\le d(e_{y_{k,j}}\xi,y_{k,j})+d(y_{k,j},y_{k,i})+d(y_{k,i},e_{y_{k,i}}\eta)\le3\rho+o(1)<a
\]
 for all $k$ large enough, so that we have a sequence of diffeomorphisms

\[
\psi_{ij,k}\eqdef e_{y_{k,i}}^{-1}\circ e_{y_{k,j}}:\;\bar{\Omega}_{\rho}\to\Omega_{a},\,k\mbox{ large enough}.
\]
 %
\begin{comment}
\[
\psi_{ij,k}\eqdef e_{y_{k,i}}^{-1}\circ e_{y_{k,j}}:\Omega_{ji,k}\to\Omega_{ij,k},\,k\in N,where\Omega_{ji,k}=\psi_{ij,k}\Omega_{a}\cap\Omega_{a},i,j\in\N_{0}.
\]
\end{comment}

By an argument combining uniform bounds on derivatives of exponential
maps on a manifold of bounded geometry, Arzela-Ascoli theorem, and
diagonalization, one may arrive at a renamed subsequence such that
for all $i,j\in\N_{0},$ $(\psi_{ij,k})_{k\in\N}$ converges in $C^{\infty}(\bar{\Omega}_{\rho})$
to some smooth function $\psi_{ij}:\bar{\Omega}_{\rho}\to\Omega_{a}$.
Setting $\Omega_{ij}\eqdef\psi_{ij}\Omega_{\rho}\cap\Omega_{\text{\ensuremath{\rho}}}$,
we have $\psi_{ij}\circ\psi_{ji}=\mathrm{id}$ on $\Omega_{ij}$ and
$\psi_{ji}\circ\psi_{ij}=\mathrm{id}$ on $\Omega_{ji}$. Therefore
$\psi_{ji}=\psi_{ij}^{-1}$ in restriction to $\Omega_{ij}$, and
$\psi_{ji}$ is a diffeomorphism between $\Omega_{ij}$ and $\Omega_{ji}$.
Note that this construction gives that $\psi_{ii}=\mathrm{id}$ ,
$\Omega_{ii}=\Omega_{\rho}$ for all $i\in\N_{0}$. Furthermore,
\[
\psi_{\ell i}=\lim_{k\to\infty}e_{y_{k,\ell}}^{-1}\circ e_{y_{k,i}}=\lim_{k\to\infty}e_{y_{k,\ell}}^{-1}\circ e_{y_{k,j}}\circ e_{y_{k,j}}^{-1}\circ e_{y_{k,i}}\]
\[=\lim_{k\to\infty}e_{y_{k,\ell}}^{-1}\circ e_{y_{k,j}}\circ\lim_{k\to\infty}e_{y_{k,j}}^{-1}\circ e_{y_{k,ji}}=\psi_{\ell j}\circ\psi_{ji},
\]
and
\[
\psi_{ij}(\Omega_{ji}\cap\Omega_{jk})=\psi_{ij}(\psi_{ji}(\Omega_{\rho})\cap\Omega_{\text{\ensuremath{\rho}}}\cap\psi_{jk}(\Omega_{\rho})\cap\Omega_{\text{\ensuremath{\rho}}})=\Omega_{ij}\cap\Omega_{ik},
\]
This allows to invoke a gluing theorem (e.g. \cite[Theorem 3.1]{Gallier3})
to conclude that exists a smooth differential manifold $M_{\infty}$
with an atlas $\lbrace\Omega_{\rho},\varphi_{i}\rbrace_{i\in\N_{0}}$
whose transition maps satisfy $\varphi_{j}^{-1}\varphi_{i}=\psi_{ij}:\Omega_{ji}\to\Omega_{ij}$. 

The manifold $M_{\infty}$ can be endowed with the following Riemannian
metric. The metric on $M$ is a bilinear form on $TM$, expressed
in the atlas $\lbrace\Omega_{\rho},e_{y}\rbrace_{y\in Y}$ as follows.
For each $y\in Y$ we have an orthonormal frame $\lbrace\nu_{\alpha}^{y}\rbrace_{\alpha=1,\dots,N}$
on $Te_{y}(\Omega_{\rho})\subset TM$ by using fixed Euclidean coordinates
for $\Omega_{\rho}$, and setting a local frame $\nu_{\alpha}^{y}\eqdef e_{y*}\partial_{\alpha}$,
$\alpha=1,\dots,N$. Then the metric tensor on the chart $(\Omega_{\rho},e_{y})$
is given by $g[\nu_{\alpha,}^{y}\nu_{\beta}^{y}]\circ e_{y}$. We
define a local frame on the chart$(\Omega_{\rho},\varphi_{i})$ of
$M_{\infty}$, $i\in\N_{0}$, by $\nu_{\alpha}^{(i)}\eqdef\varphi_{i*}\partial_{\alpha}$.
Then we set 
\begin{equation}
g[\nu_{\alpha}^{(i)},\nu_{\beta}^{(i)}]\circ\varphi_{i}\eqdef\lim_{k\to\infty}g[\nu_{\alpha,}^{y_{k,i}}\nu_{\beta}^{y_{k,i}}]\circ e_{y_{k,i}}\mbox{ on }\Omega_{\rho},\,i\in\N_{0}.\label{eq:metric}
\end{equation}
It can be shown that this defines a Riemannian metric on $M_{\infty}$
by verifying a compatibility relation for overlapping charts 
\begin{equation}
g_{\alpha\beta,i}=\sum_{\gamma,\delta=1}^{N}\partial_{\alpha}\psi_{ji}^{\gamma}\,\partial_{\beta}\psi_{ji}^{\delta}\;g_{\gamma\delta,j}\circ\psi_{ji}\mbox{ on }\Omega_{ij},i,j\in\N_{0},\label{eq:gluing-R}
\end{equation}
where $\psi_{ji}^{\alpha}$ are components of $\psi_{ji}$, i.e. $\psi_{ji}=\sum_{\alpha=1}^{N}\psi_{ji}^{\alpha}\nu_{\alpha}^{(i)}$
and $g_{\alpha\beta,i}\eqdef g[\nu_{\alpha}^{(i)},\nu_{\beta}^{(i)}]\circ\varphi_{i}$.
\begin{defn}
A manifold at infinity $M_{\infty}^{(y_{k})}$ of a manifold $M$
with bounded geometry, generated by a sequence $(y_{k})$ in $Y$,
is a differentiable manifold given by the construction above supplied
with a Riemannian atlas $\lbrace\varphi_{i},\Omega_{\rho},[g_{\alpha\beta,i}]\rbrace_{i\in\N_{0}}$,
whose transition maps are $\varphi_{j}^{-1}\varphi_{i}=\psi_{ji}:\Omega_{ij}\to\Omega_{ji}$,
where $\psi_{ji}=\lim_{k\to\infty}e_{y_{k,j}}^{-1}\circ e_{y_{k,i}}$
(as a sequence of maps $\Omega_{\rho}\to\Omega_{a}$) and $\Omega_{ij}=\psi_{ij}\Omega_{\rho}\cap\Omega_{\text{\ensuremath{\rho}}}$,
$i,j\in\N_{0}$; and the metric $\lbrace[g_{\alpha\beta,i}]\rbrace_{i\in\N_{0}}$
is given by (\ref{eq:metric}). 
\end{defn}

In terms of the definition above the argument of this subsection proves
the following statement. 
\begin{prop}
Let $M$ be a smooth manifold with bounded geometry. Then every sequence
$(y_{k})$ in $Y$ has a renamed subsequence that generates a Riemannian
manifold at infinity $M_{\infty}^{(y_{k})}$ of the manifold $M$.
\end{prop}

\end{document}